\documentclass[11pt, reqno]{amsart}
\usepackage{amsmath,amsthm,amssymb,amsfonts,color}

\usepackage[latin1]{inputenc}

\newcommand\bR{\mathbb{R}}

\newcommand\bL{\mathbb{L}}

\newcommand\cN{\mathcal{N}}

 \theoremstyle{definition}

\newtheorem{theorem}{Theorem}[section]
\newtheorem{lemma}[theorem]{Lemma}
\newtheorem{corollary}[theorem]{Corollary}

\newtheorem{definition}{Definition}[section]

\newtheorem{example}[theorem]{Example}

\theoremstyle{remark}
\newtheorem{remark}[theorem]{Remark}

\newcommand{\mysection}[1]{\section{#1}
 \setcounter{equation}{0}}

\begin{document}
   \title[Strong Feller property]{A note on the strong Feller property of  diffusion processes}

\author{Timur Yastrzhembskiy}
\email{yastr002@umn.edu}
\address{127 Vincent Hall, University of Minnesota,
 Minneapolis, MN, 55455}
  \begin{abstract}
   In this note we prove the strong Feller property of a  strong Markov quasi diffusion process corresponding to an elliptic operator
with merely bounded measurable coefficients.
We also prove H\"older continuity of  harmonic functions associated with the  quasi diffusion process
and Harnack inequality.
As an application, we show that for such diffusion processes the probabilistic definition
of a regular boundary point coincides with the 'analytic' one.
  The parabolic counterparts of these results are  presented as well.
The proofs are adaptations of arguments from \cite{KrS_79} and \cite{Kr_18}.
  \end{abstract}
  \maketitle
 
  \mysection{Introduction and Main Results}
                                                                                  \label{Section 1}
   Let $d \geq 1$,
  $\bR^d$ be the Euclidean space of points 
  $
     x = (x^1, \ldots, x^d).
  $
	Let $\mathfrak{B}^d$ be the Borel sigma-algebra on $\bR^d$.
	   For $A \in \mathfrak{B}^{d}$ 
	 let
	 $\partial A$ be its boundary,
	 $\bar A$ be its closure, 
	 and  $|A|$ be its Lebesgue measure.
       For $r > 0$
	and $x \in \bR^d$	
       denote 
       $$
                B_r  = \{y \in \bR^d : |y| < r\}, \quad B_r (x) = x + B_r. 
	$$
	For an open set $G$ we denote by $B (G)$
	the set of bounded Borel measurable functions.
	Let $C (\bR^d)$ 
	be	 the space of bounded continuous functions on $\bR^d$, 
	and
	$C^2_0 (\bR^d)$ be the space of  twice continuously differentiable functions 
	 with compact support.

 Let $\nu, K > 0$
 be numbers
and
 $
a (t, x)= (a^{i j} (t, x), i, j = 1, \ldots, d), 
$
 $	
 b (t, x) = (b^i (t, x), i = 1, \ldots, d)
$
 be Borel measurable functions such that,
 for every 
$
  (t, x) \in \bR^{d+1}, \xi \in \bR^d, i, j 
$
        \begin{equation}
                                   \label{1.2}
            \nu |\xi|^2 \leq  a^{k l}(t, x) \xi^k \xi^l \leq \nu^{-1} |\xi|^2,   
\, \, \, 
	a^{i j} \equiv a^{j i}, 
						\quad |b^i (t, x)| < K.
        \end{equation} 
	Denote
 \begin{equation}
				\label{1.1}
	       L =   a^{ ij } D_{ x^i  x^j}  + b^i  D_{x^i},
  \end{equation}	
   and let $\bL^t_{\nu, K}$ be the set of all such operators with coefficients satisfying \eqref{1.2}.

	By $N (\cdot)$ we denote a positive constant depending only on the parameters listed inside the parenthesis.
	
	\textbf{Results in the elliptic case.} First, we consider the case when $a$ and $b$ are independent of $t$.

	Let 
	$
		\Omega 
	$
	be the set of all continuous $\bR^d$-valued functions $x_{\cdot}$, 
	 $
		  \cN_t = \sigma (x_{\eta}, \eta \in [0, t]),
	 $
	  $
		\cN_{\infty} = \sigma (x_t, t \geq 0).
	  $
        Let 
	$
		X = (x_t, \infty, \cN_t, P_{x})
	$
	be a Markov process in the terminology of \cite{Dy_61}.
	We say that $X$ is a quasi diffusion process corresponding to $L$ 
	and denote it by $X (L)$
	if for any $\phi \in C^2_0 (\bR^d), t \in , x \in \bR^d$,
	\begin{equation}
				\label{1.3}
		\phi (x) = E_{x} \phi (x_t) -  E_{x} \int_0^t L  \phi (x_{\eta}) \, d\eta.  
	\end{equation}
	Here by $E_x$ we mean the integral over $\Omega$ with respect to $P_x$. 
	Note that we may replace \eqref{1.3} by the following condition:
	\begin{equation}
				\label{1.4}
		 \zeta_t =  \phi (x_t)  - \phi (x) - \int_0^t L \phi (x_{\eta}) \, d\eta 
	\end{equation}
	is a centered martingale relative to $(\cN_t, t \geq 0)$ on the probability space 
	$
		(\Omega, \cN_{\infty}, P_x).
	$
	The fact that $\xi_t$ is a martingale with zero mean  implies \eqref{1.4}.
	The other implication follows  from the Markov property.

	It is a classical fact that
	that there exists a strong Markov process $X (L)$.
	This result was first proved by N.V. Krylov in \cite{Kr_73}.
	A different proof was later given by R.F. Bass in \cite{B_98}.
	 It is shown in \cite{N_97, S_99} that such strong Markov process is generally not  unique.
	However, under certain regularity assumptions on the leading coefficients,
	 uniqueness does  hold.
	An interested reader can find some classical results in   \cite{SV, B_98} 
	 and some recent developments in \cite{Kr_19}.
	Strong Markov quasi diffusion processes corresponding to elliptic operators with discontinuous coefficients
 	arise naturally	in stochastic optimal control 
	as Markov processes controlled by an optimal Markov policy (see Remark 1.1 of \cite{Kr_86}).

	\textit{Strong Feller property and regular boundary   points of quasi diffusion processes.}

\begin{definition}	We say that a  quasi diffusion process 
  	$ X (L)$ is  strong Feller 
   if for any number $t > 0$
  and 
   $
	f \in B (\bR^d),
  $
	the function
	$
	   T_t f (x): =  E_x   f  (x_{t}) 
	$
       is   continuous  on $\bR^d$.
\end{definition}

In the sequel, the process $X (L)$ is strong Markov.

\begin{theorem}
			\label{theorem 1.1}
For any 
	$f \in B (\bR^d)$,
	$r \in (0, 1]$, 
	$t > 0$, 
	 $x_0 \in \bR^{d}$,
	and $ x_i \in \bar  B_{r} (x_0),  i = 1, 2$
 $$
		|T_t f (x_1) - T_t f (x_2)| \leq  N r^{-\alpha} |x_1 - x_2|^{\alpha} \sup_{ B_{2r} (x_0) } |T_t f|,
 $$
where $N, \alpha$ depend only on $d, \nu$, and $K$.
If $b \equiv 0$, then this conclusion holds for $r > 1$.
  \end{theorem}

	 To state the corollary of this theorem
  	  we denote for $G \subset \bR^d$
	\begin{equation}
				\label{1.6}
		\tau_G = \inf\{t \geq 0: x_t \not \in G\}, \, \inf \emptyset = \infty.
	\end{equation}
	We call $\tau_G$ the first exit time from $G$.
	It is well known that if $G$ is an open set, then $\tau_G$ is a stopping time relative to 
	$
		(\cN_t, t \geq 0).
	$
	If, in addition,  $G$ is bounded, and $x \in G$
	 then, $\tau_G < \infty$ $P_{x}$ a.s.
	(see, for example,  Proposition 1.8.2 of \cite{B_98}).
	Similarly, we define the first exit time from $G$ after $+0$ as follows:
	$$	
		\tau_{G}' = \inf\{t > 0: x_t \not \in G\}.
	$$
  	Further, in the terminology of \cite{Dy_61}
 	a point $x \in \partial G$
is called  regular if
$$
	P_x (\tau_G' = 0) = 1,
$$
and, the same point  is called almost regular
if  
$$
	\lim_{ G  \ni y \to x} E_y g (x_{\tau_G })  = g (x)
$$
for any $g \in B (\bR^d)$ that is continuous at the point $x$.
Further,
 if the coefficients $a^{i j}$ are regular enough, 
say, continuous everywhere,
then, 
 $
	E_y g (x_{\tau_G})
 $ 
coincides with the Perron solution to the Dirichlet problem
  $$
	L u = 0 \, \,   \text{a.e. in} \,  G, 
	\quad  u = g \quad \text{on} \, \partial G.
  $$
	
	The following assertions are used in \cite{KrY_17} to explain the idea behind the construction 
	of operators with discontinuous coefficients.

\begin{corollary}
Let $G \subset \bR^d$ be a bounded domain.
Then, the following assertions hold.

$(i)$ The probability
 $P_x (\tau_G' = 0)$ is either $0$ or $1$.

$(ii)$
A boundary point is regular 
  if and only if it is almost regular.
\end{corollary}

\begin{proof}

$(i)$ By Theorem \ref{theorem 1.1} the process $X (L)$ is strong Feller. 
Then, by Theorem 3.5 of \cite{Dy_61}
 $
	(x_t, \infty,  \cN_{t + }, P_x)
 $ is a Markov process.
Note that 
  $\tau_G'$ is an 
    $	
	 \cN_{0+}
   $
 measurable random variable, 
and, hence,  the claim holds by the zero-one law 
  (see Corollary 1 in Section 3 of Chapter 3 of \cite{Dy_61}) ).

$(ii)$  Since $ X (L)$ is strong Feller,  by
Theorem 13.3 of \cite{Dy_61} a regular point is almost regular.
The other implication is proved in Corollary 1 of \cite{Kr_66}. 

\end{proof}

\textit{Results for harmonic functions.}
	
	\begin{definition}
	 We say that $u$ is a harmonic function for
			$
			   X (L)
			$
		on a bounded domain $G$
                          if 
			$
			   u \in B (G),
			$
			 and  for any 
                       domain
			$
			     D \subset G,
			$
			and any 
			$
			  x \in D,
			$
			 one has 
			 $$
			  	u (x) =  E_{x} u ( x_{\tau_{ D}   }).
			$$
	\end{definition}

		\begin{example}
						\label{example 1.1}
		One common example of  harmonic functions
		is an exit distribution for the process $X$ defined as 
		$$
			\pi_G (x, A): = P_x (x_{\tau_G } \in A \cap \partial G), \, A \in \mathfrak{B}^d.
		$$
		By the strong Markov property the function $x \to \pi_{G} (x, A)$ is harmonic for $X (L)$ on $G$.
		Note that the next two results hold for this function.
		\end{example}

	\begin{theorem}
				\label{theorem 1.3}
	Let $r \in (0, 1]$,  $x_0 \in \bR^d$,
	and $u$ be a harmonic function for $X (L)$ on $B_{2r} (x_0)$.
	Then, for any $x, y \in \bar B_{r} (x_0)$
	$$
		|u (x) - u (y)| \leq N r^{-\alpha} |x - y|^{\alpha} \sup_{ B_{2r} (x_0)} |u|
	$$
	with $N, \alpha$ depending only on $d, \nu$, and $K$.
	If $b \equiv 0$, then the conclusion also holds for $r > 1$.
	\end{theorem}

	\begin{theorem}[Harnack inequality]
				\label{theorem 1.4}
	Invoke the assumptions of Theorem \ref{theorem 1.3}
	and assume additionally that $u$ is a nonnegative function on $B_{2r} (x_0)$.
	Then, 
	for any $x \in \bar B_{r/2} (x_0)$ we have
	$$
		u (x_0) \leq N (d, \nu, K) u (x).	
	$$
	If $b \equiv 0$,  the assertion also holds  for $r > 1$.
	\end{theorem}

	\begin{remark}
	Let us call the assumption that $u$ is harmonic for $X (L)$ on $D = B_r (x_0)$
	by $(A)$
	and
	the one that
	 $u \in B (D)$ and 
	   $
		u (x_{t \wedge \tau_D}), t \geq 0
	   $
	is a martingale relative to $(\cN_t, t \geq 0)$
	by $(A')$.
		
	In  \cite{B_98}
	it is shown that
	Theorems \ref{1.3} and \ref{1.4}
	holds with the assumption $(A)$ replaced by $(A')$
	and the additional constraint $b \equiv 0$ (see Theorems 5.7.5 and 5.7.6 of \cite{B_98}).

	Actually, $(A')$ and $(A)$ are very similar.
	Indeed, by Doob's sampling theorem $(A')$ implies $(A)$.
	Further, assume that $(A)$ holds
	and let 
	$G$ be an open set such that $\bar G \subset D$,
	and $\tau$ be a stopping time.
	Then, by the strong Markov property 
	$$
		u (x) = E_x E_{x_{\tau}} u (x_{\tau_G}) I_{\tau_G > \tau} +  E_x  u (x_{\tau_G}) I_{\tau_G \leq \tau}
	$$
	 $$
		= E_x u (x_{\tau})I_{\tau_G > \tau}  + E_x  u (x_{\tau_G}) I_{\tau_G \leq \tau} = E_x u (x_{\tau \wedge \tau_G}).
	 $$
	Since the equality between the extreme terms hold for any bounded stopping time, 
	$u (x_{t \wedge \tau_G}), t \geq 0$ is a martingale.
	\end{remark}

	\textbf{Results in the parabolic case.}

Let $\omega_t$ be continuous $\bR^{d+1}$-valued function, 
	and $x^0_t$ be its first component, and $x_t$ be the last $d$ components.
	Denote
	$
		\tilde \Omega = C ([0, \infty), \bR^{d+1}),
	$
	 $
		  \tilde \cN_t = \sigma (\omega_r, r \in [0, t]),
	 $
	  $
		\tilde \cN_{\infty} = \sigma (\omega_t, t \geq 0).
	  $
        Let $Y = (\omega_t, \infty, \tilde \cN_t, P_{s, x})$
	be a time-homogeneous Markov process in the terminology of \cite{Dy_61},
	where $s \in \bR, x \in \bR^d$.
	We say that  $Y (D_t + L): = Y$ is a quasi diffusion process corresponding  to $D_t + L$  if
	$$
			x^0_t = s + t, \, \,   P_{s, x} \, \text{a.s.}	,
	$$
	and for any $\phi \in C^2_0 (\bR^{d+1})$ 
	\begin{equation}
	\begin{aligned}
					\label{1.3.1}
		\phi (s, x) =& E_{s, x} \phi (\omega_t) -  E_{s, x} \int_0^t  (D_t + L)  \phi (\omega_{\eta}) \, d\eta\\
		&=  E_{s, x} \phi (s + t, x_t)  - E_{s, x} \int_0^t  (D_t + L) \phi (s + \eta, x_{\eta}) \, d\eta.
	\end{aligned}
	  \end{equation}

	One can show that a  strong Markov quasi diffusion process $Y (D_t + L)$ exists
	  by repeating  the argument of \cite{Kr_73}
	and replacing the  Alexandrov estimate  by  its parabolic counterpart (see Section 2.2 of \cite{Kr_77}).
	The existence of such process also follows from a more general result of \cite{AP_78}.

	In the sequel, we assume that $Y (D_t + L)$ is strong Markov.

\textit{Parabolic analogue of the strong Feller property.}

For $(t_i, x_i), i = 1, 2$
and $T, r > 0$
we denote 
$$
	\rho ((t_1, x_1), (t_2, x_2)) =   |t_1 - t_2|^{1/2} + |x_1 - x_2|.
$$
	$$
		 Q_{T, r} = [0, T) \times B_r,  \, \,   Q_r = Q_{r^2, r}, \, \,   Q_{r} (t, x) = (t, x) + Q_r.
        $$    
Here is the parabolic counterpart of Theorem \ref{theorem 1.1}.
This time, however, 
our result  does not imply that $Y (D_t + L)$ is strong Feller (see Remark \ref{remark 1.1}).

\begin{theorem}
			\label{theorem 1.2}
For $T \in \bR, f \in B (\bR^d)$ denote
	\begin{equation}
				\label{1.7}
		H (s, x) = E_{s, x} f (x_{T-s}), \, \, s \in (-\infty, T), x \in \bR^d.
	\end{equation}
Then, there exist constants $N, \alpha$ depending only on $d, \nu, K$ such that  
for any  
   $
	r \in (0, 1], 
		(s_0, x_0) \in \bR^{d+1}
   $,
	and any 
	$
		(s_i, x_i) \in \bar Q_r (s_0, x_0), i = 1, 2
	$,
$$
	|H (s_1, x_1) - H (s_2, x_2)| \leq N r^{-\alpha} \rho^{\alpha} ((s_1, x_1), (s_2, x_2)) \sup_{Q_{2r} (s_0, x_0) } |H|.
$$
  \end{theorem}

\begin{remark}
			\label{remark 1.1}
Let $f : \bR \to \bR$ be a function discontinuous at $1$.
Then, the  function
$(s, x) \to E_{s, x} f (\omega_1) = f (s+1)$ 
is discontinuous on $\{0\} \times \bR^d$,
and, then, 
$Y (D_t + L)$ is not a strong Feller process.
\end{remark}

The next theorem states that  $Y (D_t + L)$
is a Feller process.

\begin{theorem}
			\label{theorem  1.7}
For any $f \in C (\bR^{d+1})$
and $T > 0$
the function
$u (s, x) = E_{s, x} f (s+T, x_T)$
is of class $C (\bR^{d+1})$.
\end{theorem}

	To state the next results we introduce some notation.
	For a set 
	$
		A \subset \bR^{d+1}
	$ 
	we denote by $\partial_p A$ the parabolic boundary of $A$,
	that is, a subset of 
	$\partial A$
	 of all points $(s_0, x_0)$ such that
	there exists a function 
	$
		x_{\cdot} \in C ([0, \infty), \bR^d)
	$
	 and a number $\varepsilon \in (0, s_0]$
	such that
	$(t, x_t) \in  A$ for all $t \in [s_0 - \varepsilon, s_0)$.
        In particular,
	 $$
								   \partial_p Q_r (t, x) = ((t,  t + r^2) \times \partial B_r (x)) \cup (\{ t +  r^2 \} \times \bar B_r (x)).  
         $$

	For a  nonempty open set 
	$Q \subset \bR^{d+1}$
	and $s \in \bR$
	we denote 
	$$
		\tau (Q) = \inf\{t \geq 0: \omega_t \not \in Q\},
	$$
	 	       \begin{equation}
	\begin{aligned}
                                          \label{2.1}
                       \tau_s (Q) =  \inf\{ t \geq 0: (s + t, x_t)  \not \in  Q \}.
	\end{aligned}
               \end{equation}
	We make a few observations. 
	First, $\tau (Q)$ is a stopping time relative to $(\tilde \cN_t, t \geq 0)$,
	and 
	$
		\tau_s (Q)
	$
	 is a stopping time relative to $(\cN_t, t \geq 0)$.
	Second, 
	  for any $s, x,$
	$$
		\tau_s (Q) = \tau (Q)\, \, \,  P_{s, x} \,  \text{a.s.}
	$$
	Third, for any $(s, x) \in Q$ we have
	$$
		\tau_s (Q) = \inf\{t \geq 0: (s + t, x_t) \in \partial_p Q\}.
	$$

	  We say that $u$ is a harmonic function for
			$
			   Y (D_t + L)
			$
			on $Q$
                          if 
			$
			   u \in B (Q)
			$
			and for any domain $G$
			such that
			$G \subset Q$
			and
			$
			 (s, x) \in G,
			$
			 one has 
			 $$
			 	 u (s, x) =  E_{s, x} u (\omega_{\tau (G)}) =   E_{s, x} u (s + \tau_s (G), x_{\tau_s (G)} ).
			 $$

	\begin{example}	
					\label{example 1.2}
	The first example of a harmonic function is an exit distribution
	given by
	$$
		\pi_Q (s, x, A) = P_{s, x} ((s+ \tau_s (Q), x_{\tau_s (Q)}) \in A \cap \partial_p Q), \quad A \in \mathfrak{B}^{d+1}.
	$$
	Indeed, by the strong Markov property the function
	$
		(s, x) \to \pi_Q (s, x, A)
	$
	is harmonic for $Y (D_t + L)$ on $Q$.

	The second example is the function $H (s, x)$
	given by \eqref{1.7}.
	First, we prove that $H (s, x)$ is a measurable function.
	It suffices to show that  a function  
	$
		h (s, x, t) = E_{s, x} f (x_t)
	$
	is measurable.
	By the standard approximation argument
	we may assume that $f$ is bounded and continuous.
	Due to continuity of $x_t$,  
	$h (s, x, t)$ is a continuous function of $t$.
	On the other hand, 
	$h (s, x, t)$ is a measurable function of $(s, x)$
	because so is the transition function  
	$
		P_{s, x} (\omega_t \in A), A \in \mathfrak{B}^{d+1}.
	$
	Combining these two facts, we conclude that $h(s, x, t)$ is measurable.

	Next, 
	 we show that
	$H$ is harmonic for $Y (D_t + L)$
	on 
	$
		(T_0, T) \times D,
	$
	 where $D \subset \bR^d$
	is a bounded set.
	Let $s \in (T_0, T)$.
	By the strong Markov property for any stopping time $\tau \leq T_2 - s$ relative to $(\tilde \cN_t, t \geq 0)$
	$$
		 E_{s, x} f (x_{T - s}) = E_{s, x} E_{s + \tau, x_{\tau}} f (x_{T - s - \tau}) =  E_{s, x} H (s+\tau, x_{\tau}).
	$$  
	This implies the validity of the claim.
	
	\end{example}

	Most of the theorems stated above will be derived from the next two theorems.

	\begin{theorem}
				\label{theorem 1.5}
	Let $r \in (0, 1]$, $(s_0, x_0) \in \bR^{d+1}$, 
	and
	 $u$ be a  harmonic function for $Y (D_t + L)$ on $Q_{2r} (s_0, x_0)$.
	Then, there exist constants $N, \alpha$ depending only on $d, \nu, K$ such 
	that
	 for any   
	$
	 	(s_i, x_i) \in  \bar Q_r (t_0, x_0), i  =1, 2 
	$  
	  \begin{equation}
				\label{3.1}
		|u (s_1, x_1) - u (s_2, x_2)| \leq N r^{-\alpha} \rho^{\alpha} ((s_1, x_1), (s_2, x_2)) \sup_{Q_{2r} (s_0, x_0) } |u|.
	  \end{equation}
 	  
            \end{theorem}

	\begin{theorem}[Harnack inequality]
									\label{theorem 1.6}
	Let $(t_0, x_0) \in \bR^{d+1}$, 
	$
		r \in (0, 1], \varepsilon \in (0, 1),
	$
	 and $u$ be a nonnegative harmonic function for $Y (D_t + L)$
	 on  
	 $
		Q_{8r^2, 2r}  (t_0,  x_0).
	$
	Then, there exists a constant $N$ depending only on $d, \nu, K$ and  $\varepsilon$ such that
	$$
		u (r^2, x_0) \leq N \inf_{x \in \bar B_{ (2 - \varepsilon) r } (x_0) } u (t_0, x).
	$$
	If $b \equiv 0$, then the claim holds for $r > 1$.
	\end{theorem}

	\begin{remark}
	In the case when $a$ and $b$ are independent of $t$
	and $b \in L_d$ in \cite{Kr_20}
	N.V. Krylov proves the existence of  
	a strong Markov quasi diffusion process corresponding to $L$.
	Further, he shows that this process has 
	strong Feller property (see  Theorem 4.12 of \cite{Kr_20}).
	In fact, it is shown that $T_t f (x)$ is continuous on $(0, T) \times \bR^d$.
	Further, under the aforementioned conditions,
	in Section 6 of \cite{Kr_20}
	the author proves the appropriate versions  of Theorems \ref{theorem 1.3}, \ref{theorem 1.4},  \ref{theorem 1.5},  and \ref{theorem 1.6} of the present paper. 
	\end{remark}

	The rest of the note is organized as follows. 
	In Section \ref{section 2} we prove some lemmas including 
	 a probabilistic version of the  Krylov-Safonov estimate.
	 In Section \ref{section 3} we prove Theorem \ref{theorem 1.5} first, then we  obtain Theorems \ref{theorem 1.1}, \ref{theorem 1.2}, and  \ref{theorem 1.3} as corollaries. 
	In Section \ref{section 4}  we prove the parabolic Harnack inequality and derive the elliptic one from it. 

	Finally, this author would like to thank his advisor N.V. Krylov for the statement of the problem, useful suggestions and attention to this work.

       \mysection{Auxiliary results}
                                                                                         \label{section 2}
     
	The following lemma is a version of Theorem 4.5.1 of \cite{SV}.
	
	\begin{lemma}
				\label{lemma 2.1}
	$(i)$
	Let 
	$Y (D_t + L)$ be a quasi diffusion process.
	Then, for any $(s, x)$ 
	there exists a Wiener process $w_t$ on the probability space  
	 $
		(\Omega,  \tilde \cN_{\infty}, P_{s, x})
	 $
	 such that $x_{\cdot}$ solves the following stochastic differential equation:
	\begin{equation}
				\label{2.1.1}
		x_t = x + \int_0^t b (s + \eta, x_{\eta}) \, d\eta + \int_0^t \sqrt{2 a (s + \eta, x_{\eta}) } \, dw_{\eta} 
	\end{equation}
	on the same probability space.

	$(ii)$ Assume that $a$ and $b$ are independent of $t$
	and let $X (L)$ be a quasi diffusion process.
	Then,  on  $(\Omega, \cN_{\infty}, P_x)$
	there exists a Wiener process $w_t$
	such that  $x_{\cdot}$ satisfies
	$$
		x_t = x  + \int_0^t b (x_{\eta}) \, d\eta + \int_0^t \sqrt{2 a (x_{\eta}) } \, dw_{\eta}. 
	$$
	\end{lemma}

	\begin{proof}
	Denote
	$$
		b_{\eta} = b (s+\eta, x_{\eta}), \quad a_{\eta} = a (s + \eta, x_{\eta}).
	$$

	First, note that by the Markov property, for any $v \in C^2_0 (\bR^d)$ and \eqref{1.3.1}
	the process
	$$
		v (x_t) - v (x) -  \int_0^t L v (x_{\eta}) \, d\eta
	$$
	is a centered martingale relative to
	$
		(\tilde \cN_t, t \geq 0)
	$
	on $(\Omega, \tilde \cN_{\infty}, P_{s, x})$.
	By the standard approximation argument, for  $v = x^i$ and $v =  x^i x^j$,
	 the above process is a local martingale. 
	Then, by what was just said
	$$
		\xi^i_t: = x^i_t - \int_0^t b^i_{\eta} \, d\eta, 
	$$
	 $$
		\xi^{i j}_t : = x^i_t x^j_t - \int_0^t (b^i_{\eta}  x^j_{\eta}  + b^j_{\eta}  x^i_{\eta}) \, d\eta -  2 \int_0^t a^{i j}_{\eta}  \, d\eta
	 $$
	are local martingales.

	Next, denote
	$$
		\phi_t = \int_0^t b_{\eta}\, d\eta.
	$$
	We claim that 
	$$
		\zeta^{i j}_t: =  \xi^i_t \xi^j_t  -  2 \int_0^t a^{i j}_{\eta} \, d\eta
	$$
	is a local martingale.
	To prove this, first, we rewrite $\xi^{i j}_t$ as follows:
	$$
		\xi^{i j}_t = \xi^i_t \xi^j_t + \xi^i_t \phi^j_t + \xi^j_t \phi^i_t + \phi^i_t \phi^j_t
	$$
	 $$
		- \int_0^t (b^i_{\eta} \xi^j_{\eta}  + b^j_{\eta}  \xi^i_{\eta}) \, d\eta  - \int_0^t (b^i_{\eta} \phi^j_{\eta}  +  b^j_{\eta} \phi^i_{\eta}) \, d\eta  -   2 \int_0^t a^{i j} (s + \eta, x_{\eta}) \, d\eta.
	 $$
	Due to integration by parts the sum of the fourth and the sixth terms on the right hand side of the above expression is zero.
	Further, thanks to Lemma $3.6.15$ of \cite{Kr_95}, 
	the process 
	$$
		\xi^i_t \phi^j_t - \int_0^t \xi^i_{\eta} b^j_{\eta}  \, d\eta
	$$
	 is a local martingale. 
	Hence, $\zeta^{i j}_t$ is a  local martingales.
	By this, the fact that $\xi^i_t$ is a local martingale,
	and Theorem 3.10.8 of \cite{Kr_95}
	 $
		w_t =\int_0^t  (2a_{\eta})^{-1/2} \, d\xi_{\eta}
	$
	is a Wiener process,
	 and $P_{s, x}$ a.s. for all $t \geq 0$
	$$
		\xi_t - x  = \int_0^t  \sqrt{2 a_{\eta} } \, dw_{\eta}.
	$$
	\end{proof}

	\begin{lemma}
				\label{lemma 2.2}
	Let 
	$
		(s_0, x_0) \in \bR^{d+1}
	$,
	$r > 0$ be a number.
	Denote 
	$$
		 \hat a^{i j} (t, x) =  a^{i j} ( s_0 + r^2  t, x_0 +  r x), \quad  
		\hat b (t, x) =  r b^i ( s_0 + r^2 t, x_0 + r x)   
	$$
	and let $\hat L$ be an operator given by \eqref{1.1} with $a$ and $b$ 
	replaced by $\hat a$ and $\hat b$.
	For any open set $A \subset \bR^{d+1}$
	we denote
	$$
		\hat A = \{((t - s_0)/r^2, (x - x_0)/r): (t, x) \in A\}.
	$$
	Let
	$u$ be a harmonic  function for $Y (D_t + L)$ on $Q_r (s_0, x_0)$ 
	if and only if  
	$$
		 \hat u (s, x) = u ( s_0 +  r^2 s,  x_0  + r x) 
	$$
	is  harmonic  for some strong Markov process  $Y (D_t + \hat L)$ on $Q_1 (0, 0)$. 
	\end{lemma}

	\begin{proof}
	For any 
	$
		\omega  = (x^0_{\cdot}, x_{\cdot})\in \tilde \Omega
	$
	 we denote
	$$
		\hat x^0_t =  r^{-2} (x^0_{r^2 t} - s_0),  \quad
		\hat x_t = r^{-1} (x_{r^2 t}  -  x_0),  \quad \hat \omega_t = (\hat x^0_t, \hat x_t),
	$$
	  $$ 
	   \hat \cN_t = \sigma (\hat\omega_{\eta}, \eta \in [0, t]),
	\quad \hat P_{s, x} = P_{s_0 +  r^2 s, x_0 + r x}.
	 $$ 
	It follows that
	the process 
	$
		 \hat X := ( \hat \omega_t, \hat\cN_t, \hat P_{s, x})
	$
	is  strong Markov.
	Below we will show that $ \hat X$ is a quasi diffusion process corresponding  to $\hat L$.

	Note that by Lemma \ref{lemma 2.1}
	$\hat x_t$
	satisfies the equation
	$$
		\hat x_t =  r^{-1} (x - x_0)  + r^{-1} \int_0^{r^2 t} b (s+\eta, x_{\eta}) \, d\eta +  r^{-1} \int_0^{r^2 t} \sqrt{2 a (s + \eta, x_{\eta}}) \, dw_{\eta}
	$$
	on 
	$
		(\Omega, \tilde \cN_{\infty}, P_{s, x}).
	$
	Due to the scaling property of a Wiener process
	$
		w^r_{\eta} := r^{-1} w_{r^2 \eta}
	$
	is a Wiener process on the same probability space, and
	$$
		\hat x_t =   r^{-1} (x - x_0)  +  r \int_0^{ t} b (s+ r^2 \eta, x_{ r^2 \eta}) \, d\eta +   \int_0^{ t } \sqrt{2 a (s + r^2 \eta, x_{ r^2 \eta}}) \, dw^r_{\eta}
	$$
	 $$
		=  r^{-1} (x - x_0) + \int_0^t \hat b ( r^{-2} (s - s_0) + \eta,    \hat x_{\eta}) + \int_0^{ t } \sqrt{2  \hat a ( r^{-2} (s  - s_0) + \eta, \hat x_{\eta}}) \, dw^r_{\eta}.
	 $$
	By It\^o's formula applied to the process 
	$
		 ( r^{-2} (s - s_0) + t, \hat x_t),
	$  
	for any 
	$
		\phi \in C^2_0 (\bR^{d+1})
	$
	 and $t \geq 0$,
	 $$
		\phi ( r^{-2} (s - s_0), r^{-1} (x - x_0)) =   E_{ s, x } \phi ( r^{-2} (s - s_0) + t, \hat x_t) 
	 $$
	  $$
		-  E_{s, x} \int_0^t  (D_t + \hat L) \phi ( r^{-2} (s - s_0) + \eta, \hat x_{\eta}) \, d\eta  
	 $$ 
	  $$
		= \hat E_{ r^{-2} (s - s_0),  r^{-1} (x - x_0)} \phi (\hat \omega_t) - \hat E_{ r^{-2} (s - s_0), r^{-1} (x - x_0)} \int_0^t (D_t +  \hat L) \phi (\hat \omega_{\eta}) \, d\eta.
	  $$
	Hence, $\hat X$ is a quasi diffusion process corresponding to $\hat L$.

	Let $u$ be any harmonic function for  $X (L)$ on $Q_r (s_0, x_0)$, 
	and 
	$
		A \subset Q_r (s_0, x_0)
	$
	be an open set.
	Then, $P_{s, x}$ a.s. we have 
	$$
		\tau_s (A) : = \inf\{t \geq 0: (s + t, x_t) \not\in A\}
	$$
	$$
		 = \inf\{t \geq 0: ( r^{-2} (s - s_0) +  r^{-2} t, r^{-1} (x_t  - x_0)) \in \not\hat  A\}
	$$
	 $$
		= r^2 \inf\{t \geq 0: (  r^{-2} (s - s_0) + t, \hat x_t) \in \not\hat A\}: = r^2 \hat \tau.
	 $$
	On the other hand, 
	$\hat \tau$ 
	coincides 	
	$
		\hat P_{  r^{-2} (s - s_0), r^{-1} (x - x_0)}
	$ 
	a.s.
	 with the first exit time from $A$  for the trajectory $\hat \omega_t$.
	By what was just said 
	we get
	$$
		 E_{s, x} u ( s + \tau_s (A),  x_{\tau_s (A)})
	$$
	 $$
		= E_{s, x} \hat u (r^{-2} (s - s_0) + \hat \tau,  r^{-1} (x_{r^2 \hat \tau}  -  x_0)) 
	$$
	$$	
	= \hat E_{ r^{-2} (s - s_0), r^{-1} x} \hat u  (\hat \omega_{\hat \tau}).
	 $$
	This combined with the fact that $u$ is harmonic  for $X (L)$ on $A$ yields
	$$
		\hat u ( r^{-2} (s - s_0), r^{-1} (x - x_0)) = \hat E_{ r^{-2} (s - s_0), r^{-1} (x - x_0)} \hat u ( \hat \omega_{\hat \tau}).
	$$

	\end{proof}

	The following lemma is a version of the so-called oblique cylinder lemma (see Lemma 9.2.1 of \cite{Kr_18}).

	\begin{lemma}
				\label{lemma 2.3}
	Let $u$ be a nonnegative harmonic function on $Q_{8, 2}$.
	Let
	 $
		 r \in (0, 2], \gamma \in (0, 1],  \varepsilon, \kappa \in (0, 1),
	$
	  $
		T \in [\kappa, \kappa^{-1}]
	  $
	be numbers
	and $x_0  \in \bR^d$ be a point such that $|x_0|  \leq \kappa^{-1}$.
	 Denote
	$
		V  = \{(t, x): t \in [0, T], |x - x_0  t/T| \leq r\}
	$ 
	and assume additionally $V \subset Q_{8, 2}$.
	Let $u$ be a nonnegative harmonic function for  $Y (D_t + L)$ on $V$.
	Then, there exist constants $n, N> 0$  depending only on $d, \nu, K, \kappa$ and $\varepsilon$ such that
	$$
		\inf_{|x| \leq  (1 - \varepsilon) r }  u (0, x) 
		  \geq N \gamma^n \inf_{|x - x_0| \leq \gamma r} u (T, x).
	$$
	If $b \equiv 0$, then the estimate holds for $r > 2$.
	\end{lemma}

	\begin{proof}
	We start by making several simplifications.
	First, by Lemma \ref{lemma 2.2},
	  we may assume that  $r = 1$.	
	Second, we may assume that 
	$
		m :  = \inf_{|x - x_0| \leq \gamma} u (T, x)  = 1.
	$
	Indeed, if $m = 0$ then the desired assertion trivially holds.
	Otherwise, we replace $u$ by $u/m$.
	Third, replacing $\gamma$ by $\gamma \varepsilon/2$ we may assume that $\gamma \leq \varepsilon/2$.
	Fourth, it suffices to prove the estimate for $|x| < 1 - \varepsilon/2$,
	and, hence, we may replace $\varepsilon$ by $\varepsilon/2$.

	Next, let 
	$
		\mu = \varepsilon^2 - \gamma^2
	$, 
	$y \in B_{1 - \varepsilon}$, 
	$\xi = y - x_0$.
	Denote
	$$
		\phi (t, x) = \mu (1 - t/T) - |x - y + \xi t/T|^2 + \gamma^2,
	$$
	 $$
		Q = \{(t, x): t \in [0, T],  \phi (t, x) > 0\}.
	 $$
	Observe that $Q \subset Q_{8, 2}$ is an open set 
	 bounded by the slanted paraboloid $\phi \equiv 0$
	and two planes $t = 0$ and $t = T$.
	 The cross sections with these planes  are $B_{\varepsilon} (y)$
	and $B_{\gamma} (x_0)$ respectively.
	Next, let  $n > 0$ be a number which we will choose later and denote
	$$
		\phi (t, x) = \mu (1 - t/T) - |x - y +  \xi t/T|^2 + \gamma^2,\quad
	$$
	 $$
		v (t, x) = \phi^2 (t, x) (\mu (1 - t/T) + \gamma^2)^{-n},
	$$
	By direct calculations 
	$$ 
		(\mu (1 - t/T) + \gamma^2)^n  (D_t + L) v (t, x) = A (t) \phi^2 (t, x) 
	$$
	 $$
		- B (t, x) \phi (t, x) + 8 a^{i j} (t, x) (x^i -  y^i + \xi^i t/T) (x^j - y^j + \xi^j t/T),
	$$
	where
	$$
		A (t) =  n \mu/T  (\mu (1  - t/T) + \gamma^2)^{-1}, 
	$$
	 $$
		B (t, x) =  2 \mu/T + (4/T) \xi^i  (x^i  -  y^i +  \xi^i t/T)  
	 $$
	  $$
		 + 4 \text{tr} a + 4 b^i (t, x) (x^i - y^i + \xi^i t/T).
	  $$ 
	Since  
	$
		|\xi^i| \leq 1 - \varepsilon + \kappa^{-1}
	$,
	and 
	  $
		\phi \leq \varepsilon^2
	$ in $Q$,
	we have
	 $$
		|B (t, x)| \leq N_1 (d, \nu, K, \varepsilon, \varkappa) \, \, \forall (t, x) \in Q.
	 $$
	Then, by what was just said, if
	$t \in [0, T]$ and
	$$
		|x - y +  \xi t/T|^2 \geq    (8  \nu)^{-1} N_1 \phi   (t, x),
	$$
	then
	one has
	$
	 	(D_t + L) v (t, x) \geq 0.
	$
	In the case 
	$$
		t \in [0, T], \, 
		|x - y +  \xi t/T|^2 \leq    (8  \nu)^{-1} N_1 \phi   (t, x)
	$$
	we have
	$$
		\mu (1 - t) + \gamma^2 \leq ( N_1 (8 \nu)^{-1} + 1)  \phi (t, x),
	$$
	and, then, since $\mu > 3\varepsilon^2/4$,
	$$
		A (t, x) \phi (t, x) - B (t, x)  > (3/4)  \, n \varepsilon^2 (\varkappa + \varkappa N_1 (8 \nu)^{-1})^{-1} - N_1.
	$$ 
	The last expression is positive
	for  
	$$
		n  = 2 +  (4/3) \, \varepsilon^{-2} N_1 (\varkappa + \varkappa N_1 (8 \nu)^{-1}).
	$$
	Hence, $(D_t + L) v \geq 0$ in $Q$.
	
	Further, 
	$$
		\max_{|x - x_0| \leq \gamma} v (T, x) \leq  \gamma^{4  - 2n}, \quad v (t, x)  = 0, \, \text{if} \, \phi (t, x) = 0, t \in (0, T),
	$$
	so that 
	$
		v \leq \gamma^{4 -  2n} u
	$ 
	on 
	$
		\partial_p Q.
	$
	Then, by  It\^o's formula and the fact that $u$ is harmonic for $Y (D_t + L)$ on $Q$ we get
	$$
		 \varepsilon^{4 - 2n} = v (0, y) \leq   E_{0, y} v (\tau_0 (Q), x_{\tau_0 (Q)}) 
	$$
	 $$
		 \leq \gamma^{4 - 2n} E_{0, y} u ( \tau_0 (Q),  x_{\tau_0 (Q)} ) = \gamma^{4  - 2n} u (0, y).
	 $$	
	Replacing $n$ by $n - 2$ we prove the desired assertion.
	\end{proof}

       \begin{lemma}[A variant of Krylov-Safonov inequality]
                                     \label{lemma 2.4}
          Let
	 $r, q, \varkappa \in (0, 1]$
	 be numbers. 
	 Fix 
	some 
	$
	 (s, x)  \in  \bar Q_{q \varkappa r^2, \varkappa r}
	$
             and let
	 $\Gamma$ be any compact subset of 
	$
	    Q_r 
	$ 
	such that
         $
		|\Gamma| > q |Q_r|.
	$

	Let $a$ and $b$ be Borel functions satisfying \eqref{1.2} and
	 $x_{\cdot}$ be a solution of \eqref{2.1.1}
	 on some probability space $(\Omega, \mathcal{F}, P)$.
	Denote 
	$
		\tau_s = \tau_s (Q_r),
	$
	 $
		\gamma_s = \tau_s (Q_r  \setminus \Gamma).
	 $
	 
          Then,
         there exists a  constant 
	$
	  \delta = \delta (d, \nu, K, q, \varkappa) > 0
	$
	such that
	     $$
                      \,           P (\gamma_s  < \tau_s ) > \delta.
             $$

	If $b \equiv 0$ then the assertion holds for $r > 1$.	
       \end{lemma}

        \begin{proof}
	Let
	$\mu_{ s, x}$ 
		 be a positive finite Borel
                  measure on $Q_r$  given by
                     $$
                            \mu_{ s, x} (A) = E_{s, x} \,  \int_0^{\tau_s}  I_{( s + t, x_t) \in A} \, ds.
                     $$     
	First, we will show that 
	\begin{equation}
				\label{2.6}
                \mu_{s, x} (\Gamma) > \delta r^2
        \end{equation}
	for some $\delta > 0$.

		Let $\Lambda^t_{\nu, K} \subset \bL^t_{\nu, K}$
		 be the subset of  operators with constant coefficients.
               We consider the following Bellman's  equation:
              \begin{equation}
					\label{2.2}
                  \inf_{ L \in    \Lambda^t_{\nu, K} }  (D_t + L) v = - I_{\Gamma}  ~ \text{in} \, Q_r, 
                 																	      \quad     v  = 0 ~ \text{on} \, \partial_p Q_r.
              \end{equation}
           	By Theorem 15.1.4 of \cite{Kr_18} 
                 there exists a unique solution 
		$
		   v \in W^{1, 2}_{d+3}  (Q_r).
		$
		By the Sobolev embedding theorem
		we conclude that
		$
			v \in C (\bar Q_r).
		$ 
	
		Further, 
        	by Lemma 4.1.5 of \cite{Kr_18} there exist
		 measurable functions  
                $
		   \tilde a, \tilde b 
		$
		satisfying  \eqref{1.2}
		such that
                           \begin{equation}
                                                          \label{2.3}
                                    D_t v  + \tilde a^{ij } D_{x^i x^j} v + \tilde b^i D_{x^i} v   =   - I_{\Gamma},   \, \text{ in} \,   Q_r, 
			\quad v  = 0 ~ \text{on} \, \partial_p Q_r. 
                           \end{equation} 
		Then,  by the maximum principle
		(see Theorem 3.1.5 of \cite{Kr_18})
                        $
				v \geq 0
			$ 
			a.e.
			in
			$
			    Q_r.
			$
			Next, 
			for 
			$
			  (t, y) \in Q_1
			$ 
			we set
			$$
				\hat v_n (t, y) =  r^{-2} v_n ( r^2 t, r y),
			$$
			 $$
			     \hat a (t, y) = \tilde a ( r^2 t, r  y),
				\quad \hat b (t, y) = r \tilde b   ( r^2 t, r y)
			 $$
			and note that $\hat v_n$ satisfies Eq. \eqref{2.3}
			with $\tilde a, \tilde b$ and $Q_r$
			replaced with
			$\hat a, \hat b$ and $Q_1$
			respectively.  	
			Observe that 
			$
				\hat b 
			$ 
			satisfies \eqref{1.2}.
			Then,
                           by Lemma 9.1.4 and Theorem 9.1.2 of \cite{Kr_18} 
                         there exists a  constant 
			$
			   \delta = \delta (d,  \nu, K,  q, \varkappa) > 0
			$
			such that
                          \begin{equation}
						\label{2.4}
				v (s, x) \geq \delta  r^{2}.
			  \end{equation}
                       
           Next, 
	by  It\^o's formula (see Theorem 2.10.1 of \cite{Kr_77})
        $$
              v (s, x) =    -  E_{s, x}   \int_0^{ \tau_s } (D_t +   L)  v ( s + t, x_t) \, dt.
        $$
	Due to \eqref{2.2} we have
	\begin{equation}
				\label{2.5}
		(D_t + L) v \geq - I_{\Gamma}, \, \, \text{a.e. in} \, \, Q,
	\end{equation}
	and, then, by this and the parabolic  Alexandrov estimate (see Theorem 2.2.4 of \cite{Kr_77})
	we conclude 
	$$
		 v (s, x) \leq \mu_{s, x} (\Gamma).
	$$
	Now \eqref{2.6} follows from \eqref{2.4}.
	
  	Further, observe that 
			  $ 
			       \tau_s \leq  r^2.
			   $
			By this and \eqref{2.6} we obtain
                                   $$
                                	    \delta r^2 < \mu_{s, x} (\Gamma) 
		=  E_{ s, x } \, I_{ \gamma_s  <  \tau_s  } \, \int_0^{\tau_s}  I_{(s + t, x_t) \in \Gamma}\, dt \leq
                                             r^2 \, P_{s, x} (\gamma_s  < \tau_s),
                                     $$
			and this proves the claim.

  \end{proof}

        \mysection{Proofs of Theorems \ref{theorem 1.1}, \ref{theorem 1.3},  \ref{theorem 1.2}, \ref{theorem 1.7}, and \ref{theorem 1.5}}                                                  
        											\label{section 3}

	We prove the theorems in the following order:
	\ref{theorem 1.5}, \ref{theorem 1.2}, 
	\ref{theorem 1.7}, \ref{theorem 1.1},
	\ref{theorem 1.3}.

	\textbf{Proof of Theorem \ref{theorem 1.5}.}
	We follow the argument of \cite{KrS_79}.

	Denote 
               $$
                     \rho( (s_1, x_1), (s_2, x_2)) = |s_1 - s_2|^{1/2} + |x_1 - x_2|.
               $$
	By  Lemma \ref{lemma 2.2} we may assume that $r = 1$ and $(t_0, x_0) = (0, 0)$.

	 We will show that for any 
          $R \in (0, 1/2]$ and a cylinder 
	$
		Q_{2R} (\tilde s_0, \tilde x_0) \subset Q_2,
	$
	\begin{equation}
				\label{3.2}
		 |u  (s_1, x_1) - u (s_2, x_2)| \leq N  R^{\alpha} \sup_{Q_{2R} (\tilde s_0, \tilde x_0) } |u|, 
			 \quad  \forall (s_i, x_i) \in Q_{2R} (\tilde s_0, \tilde x_0),
	\end{equation}
	where $N$ and $\alpha$ depend only on $d, \nu, K$.
	After that one finishes the argument as follows.
	If
	$
		\rho ((s_1, x_1), (s_2, x_2)) \leq 1/2
	$
	we obtain \eqref{3.1}
	by using  \eqref{3.2}
	 with 
	$$
		R = \rho ((s_1, x_1), (s_2, x_2)), \, \, 
	\tilde s_0 = s_1 \wedge s_2, \, \, 
		\tilde x_0 = (x_1 + x_2)/2.
 	 $$
	Otherwise,
	 \eqref{3.1} trivially holds with $N = 2^{1 + \alpha}$.

	Next, by Lemma \ref{lemma 2.2}
	 we may assume that $(\tilde s_0, \tilde x_0) = 0$.
        Replacing $u$ with 
	$
	  c_1 u + c_2, c_1, c_2 \in \bR
	$
          if necessary, we may assume that                     
         \begin{equation}
                                     \label{3.3}
         \sup_{ (s, x) \in \bar Q_R } u (s, x) = 1, \quad \inf_{ (s, x) \in \bar Q_R } u (s, x) = -1.
          \end{equation}
         \begin{equation}
              \label{3.4}
         |\{ (s, x) \in \bar Q_{2R} : u (s, x) \leq 0 \}| \geq|Q_{2R}|/2.
          \end{equation}
               Let 
               $$
		  \Gamma \subset \{ (s, x) \in \bar Q_{2R} : u (s, x) \leq 0 \}
	       $$
	 	be a compact set
		such that
	    $$	 
		 |\Gamma| > (3/8) |Q_{2R}|. 
	    $$       
             Let 
	   $
		  \delta (d, \nu, K, q, \varkappa)
	   $
	 be the constant from  Lemma  \ref{lemma 2.4}
               with 
	   $
		q = 3/8, \varkappa = 3/4, r = 2R
	   $
	and let 
	   $$
	   	 \hat Q =  [0,  9 R^2/8]  \times  \bar B_{3 R/2}
	  $$
	be the cylinder $Q_{q \varkappa r^2, \varkappa r}$  from Lemma \ref{lemma 2.4}.
             Then,  there exists a point 
	     $
		  (s, x) \in \bar Q_R
	     $
		 such that
                \begin{equation}
                                           \label{3.5}
                 u (s, x) > 1 - \delta/2.
                 \end{equation}
         	Further,  for the sake of convenience,
		we denote 
		$
		   \gamma = \tau_s (Q_{2R} \setminus \Gamma).
		$
               Thanks to the fact that $u$ is harmonic for $Y (D_t + L)$ on $Q_{2R}$ we have
                   \begin{equation}
					\label{3.7}
                    u (s, x)         =  E_{ s, x } u (s + \gamma,  x_{ \gamma} ) (I_{  \gamma = \tau_s  (Q_{2R}) }    +
                        																	 I_{ \gamma <  \tau_s  (Q_{2R}) }).
                      \end{equation}
                     Since $\Gamma$ is a closed set, 
		$
		   u (s + \gamma, x_{\gamma} ) I_{ \gamma <\tau_s (Q_{2R}) } \leq 0
		$
		  $P_{s, x}$ a.s.
               Using this and \eqref{3.5}, 
                  we obtain 
                    \begin{equation}
					\label{3.6}
                          1 - \delta/2  < u (s, x) \leq   P_{ s, x } ( \gamma  = \tau_s (Q_{2R}))  \sup_{  \bar  Q_{2R} } u
										            					      \leq (1 - \delta) \sup_{ \bar   Q_{2R} } u .
                    \end{equation}
                          The last inequality follows from Lemma \ref{lemma 2.4}
                          because 
			 $
			     (s, x) \in \bar Q_R \subset   \hat Q.
			 $
                      By \eqref{3.6} combined with \eqref{3.3} we get
                       $$ 
                       \text{osc}_{\bar Q_{2R}} u \geq 
															     N\, \text{osc}_{ \bar Q_R } \, u,
                         $$
                           where 
                           $$
                                  N = \frac{1 - 3\, \delta/4} {1 - \delta} > 1, \quad \text{osc}_{ \bar Q_R } u = 2.
                            $$
			By an iteration argument this implies \eqref{3.1}.

		  \textbf{Proof of Theorem \ref{theorem 1.2}.}
		
		Theorem \ref{theorem 1.2} is a direct corollary of Example \ref{example 1.2}
		and Theorem \ref{theorem 1.5}.

		\textbf{Proof of Theorem \ref{theorem 1.7}.}
		By the continuity theorem for characteristic functions it suffices to prove that
		for any $\xi \in \bR^d$
		$$
			\phi (s, x) = E_{s, x} \exp (i \xi^j x^j_T)
		$$
		is continuous at every point $(s, x) \in \bR^{d+1}$.
		
		Let   $|s_1 - s_2|, |x_1 - x_2| < 1$.
		Denote $h = s_2 - s_1$ and assume that $h \in (0, T)$.
		
		By  the triangle inequality we only need to estimate
		$
			I_1 =   |\phi (s_1, x_1) - \phi (s_1, x_2)|
		$ 
		and 
		$
			I_2 = |\phi (s_1, x_2) - \phi (s_2, x_2)|.
		$		
		 First,  by Theorem \ref{theorem 1.2}
	 	$$
			I_1 \leq N (d, \nu, K, r)  |x_1 - x_2|^{\alpha}. 
		 $$
		
		Next, for any $(s, x)$
		$$
			\psi (s, x) = |E_{s, x} \exp (i \xi^j x^j_T) - E_{s, x} \exp ( i \xi^j x^j_{T-h})| 
		$$
		 $$
			\leq N (d, \xi) (E_{s, x} |x_T  - x_{T - h}|^2)^{1/2}.
		 $$
		By Lemma \ref{lemma 2.1} and isometry of stochastic integral 
		\begin{equation}
					\label{3.11}
			E_{s, x} |x_T  - x_{T - h}|^2 \leq N (d,  \nu, K) (h + h^2).
		\end{equation}

		Further, by the Markov property
		$$
			\phi (s_1, x) = E_{s_1, x} E_{s_2, x_h} \exp ( i \xi^j x^j_{T - h})
		$$
		and, hence, 
		$$
			I_2 \leq   \psi (s_2, x_2) + \tilde I_2,
		$$
		where
		$$
			\tilde I_2 = 
			 |E_{s_1, x_2} (E_{s_2, x_h} \exp ( i \xi^j x^j_{T - h}) - E_{s_2, x_2} \exp ( i \xi^j x^j_{T - h}))|
		 $$
		Note that by Theorem \ref{theorem 1.2}
		  $$
			\tilde I_2 \leq N (d, \nu, K, \xi) ( E_{s_1, x_2} |x_h  -  x_2|^{\alpha}  + P_{s_1, x_2} (|x_h - x_2| \geq 1)), 
		  $$
		and, then, by Chebyshov's inequality and \eqref{3.11}
		  $$
			\tilde I_2  \leq N (h^{\alpha/2} + h^{1/2} + h).
		  $$
		This combined with the estimate of $I_1$ and $\psi (s_2, x_2)$ 
		proves the validity of the assertion.

		\textbf{Proof of Theorem \ref{theorem 1.1}.}
		For $s \in [0, T)$ and $x \in \bR^d$ we denote
		$$
			u (s, x) = E_x f (x_{T - s}).
		$$
		Thanks to the strong  Markov property for any stopping time $\tau \leq T$
		 relative to 
		$(\cN_t, t \geq 0)$
		 $$
			u (s, x) =  E_x u (s + \tau, x_{\tau}).
		 $$
		Repeating  word-for-word the  argument   of the
		proof of Theorem \ref{theorem 1.5}
		we conclude that the function
		satisfies the estimate \eqref{3.1}.
		The theorem is proved.

		\textbf{Proof of Theorem \ref{theorem 1.3}.}
		Let $\tilde P_{s, x}$ be the distribution of the process 
		$
			(s + t, x_t), t \geq 0
		$
		on 
		$
			(\tilde \Omega, \tilde \cN_{\infty}).
		$
		By Lemma \ref{lemma 2.1} and It\^o's formula
		for any 
		$\phi \in C^2_0 (\bR^{d+1})$
		  \eqref{1.3.1} holds.

		Next, by 
		by Theorem \ref{theorem 1.1} the function
		$
			\tilde E_{s, x} \exp (i \xi^j x^j_T) = E_x  \exp (i \xi^j x^j_T)
		$
		is continuous on $\bR^{d+1}$ for any $\xi \in \bR^d$,
		and, hence, by the continuity theorem 
		$\tilde P_{s, x}$ is a Feller family of probability measures. 
		Then, by Theorem 3.10 of \cite{Dy_61}
		$
			\tilde Y = (\omega_t, \infty,  \tilde \cN_t, \tilde P_{s, x})
		$
		is a strong Markov process. 
		Thus, we  constructed a strong Markov quasi diffusion process $Y (D_t + L) = \tilde Y$.

	Now we derive the desired assertion from Theorem \ref{theorem 1.5}.
	Note that for any 
	$
		x \in G: =  B_{3r/2}
	$,
	 \begin{equation}
				\label{3.10}
		u (x) =  \int_{ \partial G } u (y)  \pi_{G} (x, dy),
	 \end{equation}
	where $\pi_G (x, A)$ is defined in Example \ref{example 1.1}.
	 Hence, we only need to prove the claim for
	$
		u (x) = \pi_G (x, A), \, \, A \in \mathcal{B}^d.
	$
	Let $T > 4$ and denote
	$$
		Q_T   =  [0, T) \times G, 
			\quad A_T = [0, T) \times (A \cap \partial G).  
	$$
	Due to Example \ref{example 1.2} the function
	$$
		  v_T (s, x) : = \tilde P_{s, x} ( (s+\tau_s (Q_T), x_{\tau_s (Q_T)}) \in A_T)
		 =  P_x (x_{\tau_G} \in A \cap \partial G, \tau_G < T ) 
	$$
	is   harmonic for $Y (D_t + L)$  on $Q_T$.
	Then, by Theorem \ref{theorem 1.5} 
	\begin{equation}
				\label{3.8}
		|v_T (0, x) - v_T (0, y)| \leq N |x-y|^{\alpha}, \, x, y \in \bar B_{r/2},
	\end{equation}
	where $N$ and $\alpha$ depend only on $d, \nu$ and $K$.
	Recall that 
	for any $x \in G$ we have
	 $\tau_G < \infty$ $P_{x}$ a.s., 
	and,
	then, for any $s \geq 0$, 
	\begin{equation}
				\label{3.9}
	 	\lim_{T \to \infty} v_T (s, x) = \pi_G (x, A),  \, x \in B_R.
	\end{equation}
	Passing to the limit in \eqref{3.8} we prove the assertion.

	\mysection{Proofs of Theorems  \ref{theorem 1.4} and \ref{theorem 1.6}}
									\label{section 4}

	First we prove an auxiliary result which  first assertion 
	 is a variant of Theorem 9.5.1 of \cite{Kr_18}, 
	and the second one is a version of Lemma 9.5.3 of \cite{Kr_18}. 
	We follow the argument of \cite{Kr_18} very closely.
 
	\begin{lemma}
				\label{lemma 4.1}
	Let $q \in (0, 1]$, 
	 and $U_{q} (Q_1)$ be the 
	set of nonnegative Borel measurable functions $u$ such that
	
	\begin{itemize}
	\item there exists a cylinder $ Q_{R} (s, x) \supset \bar Q_1$, 
	an operator $ L \in \bL^t_{ \nu, K }$,
	and a strong Markov process $Y (D_t + Y)$ 
	such that $u$ is harmonic for  $Y (D_t + L)$ on $Q_R (s, x)$,  

	\item  $|\{(t, x) \in Q_{1}: u (t, x) \geq 1\}| \geq q |Q_1|$.
	\end{itemize}

	We introduce a  quantity 
	$$
		p (q)  = \inf_{ u \in U_q (Q_1)} |u (0, 0)|
	$$
	which is akin to capacity from the classical potential theory.
	Then, the following assertions hold.
	
	$(i)$ $p (q) \in (0, 1]$ and $p (q) \to 1$ as $q \to 1$.

	$(ii)$
	Let $(t_i, x_i) \in \bR^{d+1}, i = 0, 1$, 
	 $r \in (0, 1], R > 0, \kappa > 0$,
	and 
	$
		\bar Q_r (t_0, x_0) \subset Q_R (t_1, t_2)
	$
	and    $L \in \bL^t_{ \nu, K}$.
	Let
	 $v$ be a nonnegative function harmonic for $X (L)$ on
	  $Q_R (t_1, x_1)$
	such that
	$
		|\{(t, x) \in Q_{r} (t_0, x_0): v (t, x) \leq \kappa\}| \geq q |Q_r|
	$.
	Then, 
	$$
		v (t_0, x_0) \leq p (q) \kappa + (1  -  p (q)) \max_{\bar Q_r (t_0, x_0)} v.
	$$

	\end{lemma}

	\begin{proof}
	$(i)$ By definition, $p (q) \geq 0$.
		Since $u \equiv 1$ belongs to $U_q (Q_1)$, 
	we have $p (q) \leq 1$.

	Fix any 
	$ u \in U_q (Q_1)$.
	To prove the second claim we only need to show that,
	for any $\varepsilon > 0$
	there exists $q$ such that
	 $u (0, 0) \geq 1 - \varepsilon$.
	We may assume that $u (0, 0) < 1$.
	Denote 
	$$
		\phi (t, x) = 1 - t - |x|^2, \quad
		\tilde Q  = \{ (t, x) \in Q_1: \phi (t, x) > u (t, x)\}.
	$$
	By Theorem \ref{theorem 1.5}
	$u$ is a  continuous function on $\bar Q_1$, 
	and, then, $\tilde Q$ is an open set.
	Since, $1 = \phi (0, 0) > u (0, 0)$, 
	we have
	$\tilde Q \neq \emptyset$.
	Further, since $\phi \leq 0$ outside 
	of $Q_1$ and $u$ is nonnegative,
	we conclude that 
	$
		\partial_p \tilde Q \subset Q_1,
	$
	and 
	\begin{equation}
					\label{4.1}
		u  = \phi \, \, \text{on}  \,  \partial_p \tilde Q.
	\end{equation}
	By this, 
	$
		\{ (t, x) \in Q_1: u (t, x) \geq 1\} \subset  Q_1 \setminus \tilde Q,
	$
	and, hence,
	$	
		|\tilde Q| \leq (1 - q)|Q_1|.
	$

	Next, note that 
	$$
		L \phi (t, x) = -1 -  2b^i (t, x) x^i  - 2 \text{tr}\,  a  \geq   - 1 - K d - 2 d \nu^{-1}, \, (t, x) \in Q_1.
	$$
	This combined with It\^o's formula yields
	\begin{equation}
				\label{4.2}
	\begin{aligned}
		1 &= \phi (0, 0) = E_{0, 0} \phi (\tau_0 (\tilde Q), x_{\tau_0 (\tilde Q)}) - E \int_0^{\tau_0 (\tilde Q) } L \phi (t, x_t) \, dt\\
		&\leq E_{0, 0} \phi (\tau_0 (\tilde Q), x_{\tau_0 (\tilde Q)}) + N E \int_0^{\tau_0 (Q_1) } I_{ (t, x_t) \in \tilde Q } \, dt.
	\end{aligned}
	 \end{equation}
	By Theorem 2.2.4 of \cite{Kr_77} the last integral is less than
	$$
		N |\tilde Q|^{1/(d+1)} \leq N (1 - q)^{1/(d+1)} |Q_1|^{1/(d+1)}.
	$$
	By \eqref{4.1} and the fact that $u$ is harmonic for $Y (D_t + L)$ on $Q_1$ we have 
	$$
		 E_{0, 0} \phi (\tau_0 (\tilde Q), x_{\tau_0 (\tilde Q)}) = u (0, 0).
	$$
	Combining this \eqref{4.2} with  we obtain
	$$
		  u (0, 0) \geq  1  - N (1 - q)^{1/(d+1)} |Q_1|^{1/(d+1)},
	$$
	and this proves the assertion $(i)$.

	$(ii)$
	By virtue of Lemma \ref{lemma 2.2} we may assume that 
	$
		r = 1, (t_0, x_0) = (0, 0)
	$, and then, 
	$v \in U_q (Q_1)$.
	Obviously, the claim holds if 
	$\kappa \geq  \max_{Q_1} v$,
	and, hence, we may assume that 
	$\kappa < \max_{Q_1} v$.
	
	Next, denote $\tilde v = (\max_{Q_1} v - \kappa)^{-1}  (\max_{Q_1} v  - v)$ 
	and observe that 
	$\tilde v$ is a nonnegative function harmonic for $Y (D_t + L)$ on $A$ such that
	$
	   \{(t, x) \in Q_{1}: v (t, x) \leq \kappa\}| = |\{(t, x) \in Q_{1}: \tilde v (t, x) \geq 1\}. 
	$
	Since $\tilde v \in U_q (Q_1)$, we have
	$\tilde v (0, 0) \geq p (q)$, 
	and this finishes the proof.

	\textbf{Proof Theorem \ref{theorem 1.6}.}
	We repeat the proof of Theorem 9.6.1 almost verbatim.
	First, we exclude the trivial case $u (4, 0) = 0$.
	By Lemma \ref{lemma 2.2}
	 we may assume that 
	$t_0 = 0, x_0 = 0$, and $r = 1$.
	Next,  let $n$ be the constant from  Lemma \ref{lemma 2.3} with 
	   $\kappa = 1/4$ and $\varepsilon \in (0, 1]$.
	By Lemma \ref{lemma 4.1}
	 it is possible to find $q \in (0, 1)$ such that
	\begin{equation}
				\label{4.3}
		p  (q) >   (2^n - 1)(2^{n} - 1/2)^{-1}.
	\end{equation}
	We denote $Q_0 (4, 0) = (4, 0)$, and for $r \in [0, 1)$ we set
	$$
		m (r) =   \max_{ (t, x) \in \bar Q_r (4, 0)} u (t, x), 
		\quad f (r) = (1 - r)^{-n} u (4, 0).
	$$
	Obviously, $f$ is continuous on $[0, 1)$,
	 and $f (r) \to \infty$ as $r \uparrow 1$.
	The function $m$ is a continuous function on $[0, 1]$, 
	because  so is $u$  on  $\bar Q_{ 3/2} (4, 0)$ thanks to Theorem \ref{theorem 1.5}.
	This combined with the fact that $m (0) = f (0)$
	implies the existence of the greatest root of the equation $m (r) = f (r)$
	which we denote by $r_0$.
	 
	Due to the  continuity of $u$, there exists a point 
	$(t_1, x_1) \in \bar Q_{r_0} (4, 0)$ 
	such that
	$m (r_0) =  u (t_1, x_1)$.
	Observe that  
	$
		Q: = \bar Q_{(1 - r_0)/2} (t_1, x_1) \subset \bar Q_{(1 + r_0)/2} (4, 0)
	$, 
	and, since 
	 $
		(1 + r_0)/2 > r_0
	 $,
	   we get
	\begin{equation}
				\label{4.5}
	  	\max_{  Q }  u < 	2^{n}  u (4, 0) (1 -r_0)^{-n} = 2^n f (r_0).
	\end{equation}
	Denote 
	$
		\Gamma = \{(t, x) \in Q: u (t, x) \geq 2^{-1}  f (r_0)\}.
	$
	We claim that 
	\begin{equation}
					\label{4.4}
		|\Gamma|  \geq (1 - q) |Q|.
	\end{equation}
	Assume that the contrary holds. 
	Since $u$ is harmonic for $Y (D_t + L)$ on $Q_{ 8, 2}$, 
	by Lemma \ref{lemma 4.1} $(ii)$ with 
	$r = (1 - r_0)/2$,
	 $
		\kappa =  2^{-1} f (r_0)
	 $
	combined with \eqref{4.5}
	 we have
	$$
		 f (r_0) =   u (t_1, x_1) \leq p (q) 2^{-1} f (r_0) + (1 - p (q)) \max_{ Q} u 
	$$
	 $$
			\leq	( p (q) 2^{-1}  +  (1 - p (q)) 2^n)  f (r_0).
	 $$
	This contradicts \eqref{4.3}, and, hence, \eqref{4.4} holds.

	Next, denote 
	$
		\tau = \tau_{t_1}  (Q_{ 8, 2 } \setminus \Gamma).
	$
	Using the fact that $u$ is a nonnegative harmonic function combined with \eqref{4.5},
	for any 
	 $
		x \in B_{(1 - r_0)/2} (x_1)
	 $
	 we get
	$$
		 u (t_1, x) \geq E_{t_1, x} u (t_1 + \tau,   x_{\tau}) I_{\tau < \tau_{t_1} (Q_{ 8, 2 }) }
	$$
	 $$
		\geq 2^{-1} f (r_0) P_{t_1, x} ( \tau  < \tau_{t_1} (Q_{ 8, 2})).
	 $$
	 Combining this with Lemmas \ref{lemma 2.1} and \ref{lemma 2.4}
	we get
	$$
		\inf_{ x \in \bar B_{(1 - r_0)/4 }   (x_1) } u (t_1, x) \geq  \delta (1 - r_0)^{-n} 2^{-1} u (4, 0).
	$$
	Now the desired assertion is derived from Lemma \ref{lemma 2.3} in the following way:
	$$
		\inf_{ x \in \bar B_{ 2 - \varepsilon} } u (0, x) \geq N  4^{-n} (1 - r_0)^n \inf_{ x \in \bar  B_{(1 - r_0)/4 }   (x_1) } u (t_1, x)
	$$
	  $$
		\geq N \delta  2^{-3n-1} u (4, 0).
	  $$
	\end{proof}

	\textbf{Proof of Theorem \ref{theorem 1.4}.}
	It follows from \eqref{3.10}
	that it suffices to prove the theorem for 
	$
		u (x) = \pi_{ G } (x, A), 
	$
	$
		G = B_{3r/2},
	$
	 $
		A \in \mathcal{B}^d.
	 $
	
	Next, let $T > 2 (3r/2)^2$ be a number and 
	 d $v_T (s, x)$ be a
	 function defined
	in the proof of Theorem \ref{theorem 1.3}. 
	Observe that
	by Theorem \ref{theorem 1.6} we have
	$$
		v_T (r^2, x_0) \leq N (d, \nu, K)   \inf_{ x \in \bar B_{r/4} (x_0) }  v_T (0, x).
	$$
	 Now the theorem follows from \eqref{3.9}.

  \end{document}